\documentclass[12pt, reqno]{amsart}
\usepackage{amsmath, amsthm, amscd, amsfonts, amssymb, graphicx, color}
\usepackage[bookmarksnumbered, colorlinks, plainpages]{hyperref}
\hypersetup{colorlinks=true,linkcolor=red, anchorcolor=green, citecolor=cyan, urlcolor=red, filecolor=magenta, pdftoolbar=true}

\newtheorem{theorem}{Theorem}[section]
\newtheorem{lemma}[theorem]{Lemma}

\newtheorem{corollary}[theorem]{Corollary}
\theoremstyle{definition}

\theoremstyle{remark}
\newtheorem{remark}[theorem]{Remark}
\numberwithin{equation}{section}

\newcommand{\be}{\begin{equation}}
\newcommand{\ee}{\end{equation}}

\newcommand{\cM}{{\mathcal M}}
\newcommand{\NN}{\mathbb{N}}

\begin{document}
\setcounter{page}{1}

\title[Inequalities on the joint and generalized spectral radius]{Inequalities on the joint and generalized spectral and essential spectral radius  of the Hadamard geometric mean of bounded sets of positive kernel operators}

\author[A. Peperko]{Aljo\v{s}a Peperko$^{1,2} $}

\address{$^{1}$  Faculty of Mechanical Engineering, University of Ljubljana, A\v{s}ker\v{c}eva 6, SI-1000 Ljubljana, Slovenia;
\newline
$^{2}$ Institute of Mathematics, Physics, and Mechanics,
Jadranska 19, SI-1000 Ljubljana, Slovenia.}
\email{\textcolor[rgb]{0.00,0.00,0.84}{aljosa.peperko@fmf.uni-lj.si; aljosa.peperko@fs.uni-lj.si}}





\begin{abstract}
 Let $\Psi$ and $\Sigma$ be bounded sets of positive kernel operators on a Banach function space $L$. We prove several refinements of the known inequalities  
$$\rho \left(\Psi  ^{\left( \frac{1}{2} \right)} \circ \Sigma  ^{\left( \frac{1}{2} \right)} \right) \le \rho (\Psi  \Sigma) ^{\frac{1}{2}} \;\; \mathrm{and}\;\; \hat{\rho} \left(\Psi  ^{\left( \frac{1}{2} \right)} \circ \Sigma  ^{\left( \frac{1}{2} \right)} \right) \le \hat{\rho} (\Psi  \Sigma) ^{\frac{1}{2}} $$
for the generalized spectral radius $\rho$ and the joint spectral radius 
$\hat{\rho}$, where $\Psi  ^{\left( \frac{1}{2} \right)} \circ \Sigma  ^{\left( \frac{1}{2} \right)} $ denotes the 
 Hadamard (Schur) geometric mean of the sets $\Psi $ and $\Sigma$. Furthermore, we prove that analogous inequalities hold also for  the generalized essential  spectral radius and  the joint essential spectral radius in the case when $L$ and its Banach dual $L^*$ have order continuous norms.
\end{abstract} \maketitle

\section{Introduction}

In \cite{Zh09}, X. Zhan conjectured that, for non-negative $n\times n$ matrices $A$ and $B$, the spectral radius $\rho (A\circ B)$ of the Hadamard product satisfies
\be
\rho (A\circ B) \le \rho (AB),
\label{qu}
\ee
where $AB$ denotes the usual matrix product of $A$ and $B$. This conjecture was confirmed by K.M.R. Audenaert in \cite{Au10} 
by proving 
\be
\rho (A\circ B) \le \rho ((A\circ A)(B\circ B)) ^{\frac{1}{2}}\le \rho (AB).
\label{Aud}
\ee
These inequalities were established via a trace description of the spectral radius.
Soon after, inequality (\ref{qu}) was reproved, generalized and refined in different ways by 
several authors (\cite{HZ10},  \cite{Hu11}, \cite{S11}, \cite{Sc11}, \cite{P12}, \cite{CZ15}, \cite{DP16}, \cite{P17}, \cite{P17+}).
Using the fact that the Hadamard product is a principal submatrix of 
the Kronecker product, R.A. Horn and F. Zhang  proved in \cite{HZ10} the inequalities
\be
\rho (A\circ B) \le \rho (AB\circ BA)^{\frac{1}{2}}\le \rho (AB).
\label{HZ}
\ee
Applying the techniques of \cite{HZ10}, Z.  Huang proved that 
\be
\rho (A_1 \circ A_2 \circ \cdots \circ A_m) \le \rho (A_1 A_2 \cdots A_m)
\label{Hu}
\ee
for $n\times n$ non-negative matrices $A_1, A_2, \cdots, A_m$ (see \cite{Hu11}).  A.R. Schep was the first one to observe that the results from \cite{DP05} and \cite{P06} are applicable in this context (see \cite{S11} and \cite{Sc11}). He extended inequalities (\ref{Aud}) and (\ref{HZ}) to non-negative matrices that define bounded 
operators on sequence spaces (in particular on $l^p$ spaces, $1\le p <\infty$) and proved in 
\cite[Theorem 2.7]{S11}  that
\be
\rho (A\circ B) \le \rho ((A\circ A)(B\circ B))^{\frac{1}{2}}\le  \rho (AB\circ AB)^{\frac{1}{2}}\le \rho (AB)
\label{Sproved}
\ee
(note that there is an error in the statement of \cite[Theorem 2.7]{S11}, which was corrected in \cite{Sc11} and \cite{P12}).
In  \cite{P12}, the author of the current paper extended the inequality (\ref{Hu}) to non-negative matrices that define bounded 
operators on Banach sequence spaces (see  \cite{P12} for exact definitions) and proved that the inequalities
\be
\rho (A\circ B) \le \rho ((A\circ A)(B\circ B))^{\frac{1}{2}}\le \rho(AB \circ AB)^{\frac{\alpha}{2}} \rho(BA \circ BA)^{\frac{1-\alpha}{2}} \le \rho (AB)
\label{P1}
\ee
and
\be
\rho (A\circ B) \le \rho (AB\circ BA)^{\frac{1}{2}}\le  \rho(AB \circ AB)^{\frac{1}{4}} \rho(BA \circ BA)^{\frac{1}{4}} \le  \rho (AB).
\label{P2}
\ee
hold, where $\alpha \in [0,1]$. Moreover, 
 he generalized these inequalities to the setting of the generalized and the joint spectral radius of bounded sets of such non-negative matrices. 


In \cite[Theorem 2.8]{S11},  A.R. Schep proved
 that the inequality
\be
\rho \left(A ^{\left( \frac{1}{2} \right)} \circ B  ^{\left( \frac{1}{2} \right)} \right) \le \rho (AB) ^{\frac{1}{2}}
\label{Schep}
\ee
holds for positive kernel operators on $L^p$ spaces. Here $A ^{\left( \frac{1}{2} \right)} \circ B  ^{\left( \frac{1}{2}\right)} $ denotes the Hadamard geometric mean of operators $A$ and $B$. In \cite[Theorem 3.1]{DP16}, R. Drnov\v{s}ek and the author, generalized this inequality and proved that the inequality 
\be 
\rho \left(A_1^{\left(\frac{1}{m}\right)} \circ A_2^{\left(\frac{1}{m}\right)} \circ \cdots \circ 
A_m^{\left(\frac{1}{m}\right)}\right)   \le \rho (A_1 A_2 \cdots A_m)^{\frac{1}{m}} 
\label{genHuBfs}
\ee
holds for positive kernel operators $A_1, \ldots, A_m$ on an arbitrary Banach function space. In \cite{P17+}, the author refined (\ref{genHuBfs}) and showed that the inequalities
$$\rho \left(A_1^{\left(\frac{1}{m}\right)} \circ A_2^{\left(\frac{1}{m}\right)} \circ \cdots \circ 
A_m^{\left(\frac{1}{m}\right)}\right) \;\;\;\;\;\;\;\;\;\;\;\;\;\;\;\;\;\;\;\;\;\;\;\;\;\;\;\;\;\;$$
\be
\le \rho  \left(P_1^{\left(\frac{1}{m}\right)} \circ P_2^{\left(\frac{1}{m}\right)} \circ \cdots \circ 
P_m^{\left(\frac{1}{m}\right)}\right) ^{\frac{1}{m}}  \le \rho (A_1 A_2 \cdots A_m) ^{\frac{1}{m}} .  
\label{genHuBfsP}
\ee
hold, where  $P_j = A_j \ldots A_m A_1 \ldots A_{j-1}$ for $j=1,\ldots , m$. Formally, here and throughout the article $A_{j-1}=I$ for $j=1$ (eventhough $I$ might not be a well defined kernel operator). In particular, the following kernel version of  (\ref{HZ}) holds:
\be
\rho \left(A^{(\frac{1}{2})} \circ B^{(\frac{1}{2})} \right) \le  \rho \left((AB)^{(\frac{1}{2})} \circ (BA)^{(\frac{1}{2})}\right)^{\frac{1}{2}} \le \rho(AB) ^{\frac{1}{2}}.
\label{sch_refin}
\ee
In \cite[Theorem 3.4]{P17}, the author generalized the inequality (\ref{genHuBfs}) to the setting of the generalized and the joint spectral radius of bounded sets of positive kernel operators on a Banach function space (see also Theorem \ref{genHuDP} below). As already pointed out in \cite[Remark 3.4]{P17+}, the inequalities (\ref{genHuBfsP}) can also be deduced from the proof of \cite[Theorem 3.4]{P17}.

In this article we prove the kernel versions of all the above matrix inequalities and obtain additional refinements (even in the matrix case). Moreover, by proving the kernel versions of  results from \cite{P12} we generalize these inequalities to the setting of \cite{P17} and observe that analogous inequalities hold also for the essential spectral radius under suitable assumptions. 

The rest of the article is organized as follows. In Section 2 we recall definitions and results that we will use in our proofs and we present our results in Section 3.
In one of our main results (Theorem \ref{refin}) we generalize inequalities (\ref{sch_refin})  to the setting of the generalized and the joint spectral radius of bounded sets of positive kernel operators on an arbitrary Banach function space and give additional refinements in the sense of (\ref{P1}) and (\ref{P2}). Moreover, we prove that analogous results hold also for the generalized essential and the joint essential spectral radius of bounded sets of positive kernel operators on a Banach function space $L$ such that $L$ and $L^*$ have order continuous norms. We also point out in Theorem \ref{essPji} that under these conditions an analogue of (\ref{genHuBfsP}) for the essential radius holds. We give additional refinements 
in Corollary \ref{additional_powers}.


\section{Preliminaries}
\vspace{1mm}

Let $\mu$ be a $\sigma$-finite positive measure on a $\sigma$-algebra $\cM$ of subsets of a non-void set $X$.
Let $M(X,\mu)$ be the vector space of all equivalence classes of (almost everywhere equal)
complex measurable functions on $X$. A Banach space $L \subseteq M(X,\mu)$ is
called a {\it Banach function space} if $f \in L$, $g \in M(X,\mu)$,
and $|g| \le |f|$ imply that $g \in L$ and $\|g\| \le \|f\|$. Throughout the article, it is assumed that  $X$ is the carrier of $L$, that is, there is no subset $Y$ of $X$ of 
 strictly positive measure with the property that $f = 0$ a.e. on $Y$ for all $f \in L$ (see \cite{Za83}).

Standard examples of Banach function spaces are Euclidean spaces,  the space $c_0$ 
of all null convergent sequences  (equipped with the usual norms and the counting measure), the
well-known spaces $L^p (X,\mu)$ ($1\le p \le \infty$) and other less known examples such as Orlicz, Lorentz,  Marcinkiewicz  and more general  rearrangement-invariant spaces (see e.g. \cite{BS88}, \cite{CR07}, \cite{KM99} and the references cited there), which are important e.g. in interpolation theory and in the theory of partial differential equations.
 Recall that the cartesian product $L=E\times F$ 
of Banach function spaces is again a Banach function space, equipped with the norm
$\|(f, g)\|_L=\max \{\|f\|_E, \|g\|_F\}$.

If $\{f_n\}_{n\in \mathbb{N}} \subset M(X,\mu)$ is a decreasing sequence and
$f=\inf\{f_n \in M(X,\mu): n \in \NN \}$, then we write $f_n \downarrow f$. 
A Banach function space $L$ has an {\it order continuous norm}, if $0\le f_n \downarrow 0$
implies $\|f_n\|_L \to 0$ as $n \to \infty$. It is well known that spaces $L^p  (X,\mu)$, $1\le p< \infty$, have order continuous
norm. Moreover, the norm of any reflexive Banach function space is
order continuous. 
In particular, we will be interested in  Banach function spaces $L$ such that $L$ and its Banach dual space $L^*$ have order continuous norms. Examples of such spaces are $L^p  (X,\mu)$, $1< p< \infty$, while the space
$L=c_0$ 
is an example of a non-reflexive Banach function space, such that $L$ and  $L^*=l^1$ have order continuous
norms.

By an {\it operator} on a Banach function space $L$ we always mean a linear
operator on $L$.  An operator $A$ on $L$ is said to be {\it positive} 
if it maps nonnegative functions to nonnegative ones, i.e., $AL_+ \subset L_+$, where $L_+$ denotes the positive cone $L_+ =\{f\in L : f\ge 0 \; \mathrm{a.e.}\}$.
Given operators $A$ and $B$ on $L$, we write $A \ge B$ if the operator $A - B$ is positive.

Recall that a positive  operator $A$ is always bounded, i.e., its operator norm
\be
\|A\|=\sup\{\|Ax\|_L : x\in L, \|x\|_L \le 1\}=\sup\{\|Ax\|_L : x\in L_+, \|x\|_L \le 1\}
\label{equiv_op}
\ee
is finite.  
Also, its spectral radius $\rho (A)$ is always contained in the spectrum.

An operator $A$ on a Banach function space $L$ is called a {\it kernel operator} if
there exists a $\mu \times \mu$-measurable function
$a(x,y)$ on $X \times X$ such that, for all $f \in L$ and for almost all $x \in X$,
$$ \int_X |a(x,y) f(y)| \, d\mu(y) < \infty \ \ \ {\rm and} \ \ 
   (Af)(x) = \int_X a(x,y) f(y) \, d\mu(y)  .$$
One can check that a kernel operator $A$ is positive iff 
its kernel $a$ is non-negative almost everywhere. 

Let $L$ be a Banach function space such that $L$ and $L^*$ have order
continuous norms and let $A$ and $B$ be  positive kernel operators on $L$. By $\beta (A)$ we denote the Hausdorff measure of 
non-compactness of $A$, i.e., 
$$\beta (A) = \inf\left\{ \delta >0 : \;\; \mathrm{there}\;\; \mathrm{is} \;\; \mathrm{a}\;\; \mathrm{finite}\;\; M \subset L \;\;\mathrm{such} \;\; \mathrm{that} \;\; A(D_L) \subset M + \delta D_L  \right\},$$
where $D_L =\{f\in L : \|f\|_L \le 1\}$. Then $\beta (A) \le \|A\|$, $\beta (A+B) \le \beta (A) + \beta (B)$, $\beta(AB) \le \beta (A)\beta (B)$ and $\beta (\alpha A) =\alpha \beta (A)$ for $\alpha \ge 0$. Also 
$0 \le A\le B$  implies $\beta (A) \le \beta (B)$ (see e.g. \cite[Corollary 4.3.7 and Corollary 3.7.3]{Me91}). Let $\rho _{ess} (A)$ denote the essential spectral radius of $A$, i.e., the spectral radius of the Calkin image of $A$ in the Calkin algebra. Then 
\be
 \rho _{ess} (A) =\lim _{j \to \infty} \beta (A^j)^{1/j}=\inf _{j \in \NN} \beta (A^j)^{1/j} 
\label{esslim=inf}
\ee
and $\rho _{ess} (A) \le \beta (A)$. Note that (\ref{esslim=inf}) is valid for any bounded operator $A$ on a given complex Banach space $L$ (see e.g. \cite[Theorem 4.3.13]{Me91}).


Observe that (finite or infinite) non-negative matrices, that define operators on Banach sequence spaces, are a special case of positive kernel operators 
(see e.g. \cite{P12}, \cite{DP16}, \cite{DP10}, \cite{P11} and the references cited there).  It is well-known that kernel operators play a very important, often even central, role in a variety of applications from differential and integro-differential equations, problems from physics 
(in particular from thermodynamics), engineering, statistical and economic models, etc (see e.g. \cite{J82}, \cite{BP03}, \cite{LL05}, \cite{DLR13} 
and the references cited there).
For the theory of Banach function spaces and more general Banach lattices we refer the reader to the books \cite{Za83}, \cite{BS88}, \cite{AA02}, \cite{AB85}, \cite{Me91}. 

Let $A$ and $B$ be positive kernel operators on a Banach function space $L$ with kernels $a$ and $b$ respectively,
and $\alpha \ge 0$.
The \textit{Hadamard (or Schur) product} $A \circ B$ of $A$ and $B$ is the kernel operator
with kernel equal to $a(x,y)b(x,y)$ at point $(x,y) \in X \times X$ which can be defined (in general) 
only on some order ideal of $L$. Similarly, the \textit{Hadamard (or Schur) power} 
$A^{(\alpha)}$ of $A$ is the kernel operator with kernel equal to $(a(x, y))^{\alpha}$ 
at point $(x,y) \in X \times X$ which can be defined only on some order ideal of $L$.

Let $A_1 ,\ldots, A_n$ be positive kernel operators on a Banach function space $L$, 
and $\alpha _1, \ldots, \alpha _n$ positive numbers such that $\sum_{j=1}^n \alpha _j = 1$.
Then the {\it  Hadamard weighted geometric mean} 
$A = A_1 ^{( \alpha _1)} \circ A_2 ^{(\alpha _2)} \circ \cdots \circ A_n ^{(\alpha _n)}$ of 
the operators $A_1 ,\ldots, A_n$ is a positive kernel operator defined 
on the whole space $L$, since $A \le \alpha _1 A_1 + \alpha _2 A_2 + \ldots + \alpha _n A_n$ by the inequality between the weighted arithmetic and geometric means. Let us recall  the following result which was proved in \cite[Theorem 2.2]{DP05} and 
\cite[Theorem 5.1 and Example 3.7]{P06} (see also e.g. \cite[Theorem 2.1]{P17}).

\begin{theorem} 
Let $\{A_{i j}\}_{i=1, j=1}^{k, m}$ be positive kernel operators on a Banach function space $L$.
If $\alpha _1$, $\alpha _2$,..., $\alpha _m$ are positive numbers  
such that $\sum_{j=1}^m \alpha _j = 1$, then the positive kernel operator
$$A:= \left(A_{1 1}^{(\alpha _1)} \circ \cdots \circ A_{1 m}^{(\alpha _m)}\right) \ldots \left(A_{k 1}^{(\alpha _1)} \circ \cdots \circ A_{k m}^{(\alpha _m)} \right)$$
satisfies the following inequalities
\begin{eqnarray}
\label{basic2}
A &\le &  
(A_{1 1} \cdots  A_{k 1})^{(\alpha _1)} \circ \cdots 
\circ (A_{1 m} \cdots A_{k m})^{(\alpha _m)} , \\
\label{norm2}
\left\|A \right\| &\le &  
\|A_{1 1} \cdots  A_{k 1}\|^{\alpha _1} \cdots \|A_{1 m} \cdots A_{k m}\|^{\alpha _m}, \\ 
\label{spectral2}
\rho \left(A \right) &\le  &
\rho \left( A_{1 1} \cdots  A_{k 1} \right)^{\alpha _1} \cdots 
\rho \left( A_{1 m} \cdots A_{k m}\right)^{\alpha _m} .
\end{eqnarray}
If, in addition, $L$ and $L^*$ have order continuous norms, then
\begin{eqnarray}
\label{meas_noncomp}
\beta (A) &\le &  
\beta (A_{1 1} \cdots  A_{k 1})^{\alpha _1} \cdots \beta(A_{1 m} \cdots A_{k m})^{\alpha _m}, \\ 
\label{ess_spectral}
\rho _{ess} \left(A \right) &\le  &
\rho _{ess} \left( A_{1 1} \cdots  A_{k 1} \right)^{\alpha _1} \cdots 
\rho _{ess} \left( A_{1 m} \cdots A_{k m}\right)^{\alpha _m} .
\end{eqnarray}

\label{DPBfs}
\end{theorem}
The following result is a special case  of Theorem \ref{DPBfs}. 
\begin{theorem} 
\label{special_case}
Let $A_1 ,\ldots, A_m$ be positive kernel operators on a Banach function space  $L$,
and $\alpha _1, \ldots, \alpha _m$ positive numbers such that $\sum_{j=1}^m \alpha _j = 1$.
Then we have
\be
 \|A_1 ^{( \alpha _1)} \circ A_2 ^{(\alpha _2)} \circ \cdots \circ A_m ^{(\alpha _m)} \| \le
  \|A_1\|^{ \alpha _1}  \|A_2\|^{\alpha _2} \cdots \|A_m\|^{\alpha _m}  
\label{gl1nrm}
\ee
and
\be
 \rho(A_1 ^{( \alpha _1)} \circ A_2 ^{(\alpha _2)} \circ \cdots \circ A_m ^{(\alpha _m)} ) \le
\rho(A_1)^{ \alpha _1} \, \rho(A_2)^{\alpha _2} \cdots \rho(A_m)^{\alpha _m} .
\label{gl1vecr}
\ee
If, in addition, $L$ and $L^*$ have order continuous norms, then
\be
 \beta (A_1 ^{( \alpha _1)} \circ A_2 ^{(\alpha _2)} \circ \cdots \circ A_m ^{(\alpha _m)} )\le
  \beta(A_1)^{ \alpha _1}  \beta(A_2)^{\alpha _2} \cdots \beta(A_m)^{\alpha _m}  
\label{gl1meas_nonc}
\ee
and
\be
 \rho _{ess}(A_1 ^{( \alpha _1)} \circ A_2 ^{(\alpha _2)} \circ \cdots \circ A_m ^{(\alpha _m)} ) \le
\rho _{ess }(A_1)^{ \alpha _1} \, \rho _{ess}(A_2)^{\alpha _2} \cdots \rho _{ess}(A_m)^{\alpha _m} .
\label{gl1vecr}
\ee

\end{theorem}

Recall also that the above results on the spectral radius and operator norm  remain valid under the less restrictive assumption $\sum_{j=1}^m \alpha _j \ge 1$ in the case of  (finite or infinite) non-negative matrices that define operators on sequence spaces 
(\cite{EJS88}, \cite{DP05}, \cite{P06}, \cite{P11}, \cite{P12}, \cite{DP16}).

Let $\Sigma$ be a bounded set of bounded operators on a complex Banach space $L$.
For $m \ge 1$, let 
$$\Sigma ^m =\{A_1A_2 \cdots A_m : A_i \in \Sigma\}.$$
The generalized spectral radius of $\Sigma$ is defined by
\be
\rho (\Sigma)= \limsup _{m \to \infty} \;[\sup _{A \in \Sigma ^m} \rho (A)]^{1/m}
\label{genrho}
\ee
and is equal to 
$$\rho (\Sigma)= \sup _{m \in \NN} \;[\sup _{A \in \Sigma ^m} \rho (A)]^{1/m}.$$
The joint spectral radius of $\Sigma$ is defined by
\be
\hat{\rho}  (\Sigma)= \lim _{m \to \infty}[\sup _{A \in \Sigma ^m} \|A\|]^{1/m}.
\label{BW}
\ee
Similarly, the generalized essential spectral radius of $\Sigma$ is defined by
\be
\rho _{ess} (\Sigma)= \limsup _{m \to \infty} \;[\sup _{A \in \Sigma ^m} \rho _{ess} (A)]^{1/m}
\label{genrhoess}
\ee
and is equal to 
$$\rho _{ess} (\Sigma)= \sup _{m \in \NN} \;[\sup _{A \in \Sigma ^m} \rho _{ess} (A)]^{1/m}.$$
The joint essential  spectral radius of $\Sigma$ is defined by
\be
\hat{\rho} _{ess}  (\Sigma)= \lim _{m \to \infty}[\sup _{A \in \Sigma ^m} \beta (A)]^{1/m}.
\label{jointess}
\ee

It is well known that $\rho (\Sigma)= \hat{\rho}  (\Sigma)$ for a precompact nonempty set $\Sigma$ of compact operators on $L$ (see e.g. \cite{ShT00}, \cite{ShT08}, \cite{Mo}), 
in particular for a bounded set of complex $n\times n$ matrices (see e.g. \cite{BW92}, \cite{E95}, \cite{SWP97}, \cite{Dai11}, \cite{MP12}).
This equality is called the Berger-Wang formula or also the 
generalized spectral radius theorem (for an elegant proof in the finite dimensional case see \cite{Dai11}).
It is perhaps less well known that also the generalized Berger-Wang formula holds, i.e, that for any precompact nonempty  set $\Sigma$ of bounded operators on $L$ we have
$$\hat{\rho}  (\Sigma) = \max \{\rho (\Sigma), \hat{\rho} _{ess}  (\Sigma)\}$$
(see e.g.  \cite{ShT08}, \cite{Mo},  \cite{ShT00}). Observe also that it was proved in \cite{Mo} that in the definition of  $\hat{\rho} _{ess}  (\Sigma)$ one may replace the Haussdorf measure of noncompactness by several other seminorms, for instance it may be replaced by the essential norm.

In general $\rho (\Sigma)$ and $\hat{\rho}  (\Sigma)$ may differ even in the case of a bounded set $\Sigma$ of compact positive operators on $L$ (see \cite{SWP97} or also \cite{P17}).
Also, in \cite{Gui82} the reader can find an example of two positive non-compact weighted shifts $A$ and $B$ on $L=l^2$ such that $\rho(\{A,B\})=0 < \hat{\rho}(\{A,B\})$. As already noted in \cite{ShT00} also $\rho _{ess} (\Sigma)$ and $\hat{\rho} _{ess}  (\Sigma)$ may in general be different.

The theory of the generalized and the joint spectral radius has many important applications for instance to discrete and differential inclusions, 
wavelets, invariant subspace theory
(see e.g. \cite{BW92}, \cite{Dai11}, \cite{Wi02}, \cite{ShT00}, \cite{ShT08} and the references cited there).
In particular, $\hat{\rho} (\Sigma)$ plays a central role in determining stability in convergence properties of discrete and differential inclusions. In this 
theory the quantity $\log \hat{\rho} (\Sigma)$ is known as the maximal Lyapunov exponent (see e.g. \cite{Wi02}).

We will  use the following well known facts that hold for all $r \in \{\rho,  \hat{\rho}, \rho _{ess}, \hat{\rho} _{ess}  \}$: 
$$r (\Sigma  ^m) = r (\Sigma)^m \;\;\mathrm{and}\;\; 
r (\Psi \Sigma) = r (\Sigma\Psi) $$
where $\Psi \Sigma =\{AB: A\in \Psi, B\in \Sigma\}$ and $m\in \NN$.

Let $\Psi _1, \ldots , \Psi _m$ be bounded sets of positive kernel operators on a Banach function space $L$ and let $\alpha _1, \ldots \alpha _m$ be positive numbers such that 
$\sum _{i=1} ^m \alpha _i = 1$. Then the bounded set of positive kernel operators on $L$, defined by
$$\Psi _1 ^{( \alpha _1)} \circ \cdots \circ \Psi _m ^{(\alpha _m)}=\{ A_1 ^{( \alpha _1)} \circ \cdots \circ A _m ^{(\alpha _m)}: A_1\in \Psi _1, \ldots, A_m \in \Psi _m \},$$
is called the {\it weighted Hadamard (Schur) geometric mean} of sets $\Psi _1, \ldots , \Psi _m$. The set  
$\Psi _1 ^{(\frac{1}{m})} \circ \cdots \circ \Psi _m ^{(\frac{1}{m})}$ is called the  {\it Hadamard (Schur) geometric mean} of sets $\Psi _1, \ldots , \Psi _m$.

The folowing result (\cite[Theorem 3.3]{P17}; see also \cite{P12}, \cite{P06}) follows from   Theorem \ref{DPBfs}.


\begin{theorem}
Let $\Psi _1, \ldots \Psi _m$ be bounded sets of positive kernel operators on a Banach function space $L$ and let 
 $\alpha _1, \ldots \alpha _m$ be positive numbers such that \\
$\sum _{i=1} ^m \alpha _i = 1$.
Then we have 
\be
\rho (\Psi _1 ^{( \alpha _1)} \circ \cdots \circ \Psi _m ^{(\alpha _m)} ) \le 
\rho (\Psi _1)^{ \alpha _1} \, \cdots \rho(\Psi _m)^{\alpha _m} 
\label{gsh}
\ee
and
\be
\hat{\rho } (\Psi _1 ^{( \alpha _1)} \circ \cdots \circ \Psi _m ^{(\alpha _m)} ) \le 
\hat{\rho }  (\Psi _1)^{ \alpha _1} \, \cdots \hat{\rho } (\Psi _m)^{\alpha _m}. 
\label{gsh2}
\ee
\label{family}
\end{theorem}
The following result was the main result of \cite{P17} (see  \cite[Theorem 3.4]{P17}).
\begin{theorem} Let $\Psi _1, \ldots \Psi _m$ be bounded sets of positive kernel operators on a Banach function space $L$.
Then we have
\be
\rho \left(\Psi _1 ^{\left( \frac{1}{m} \right)} \circ \cdots \circ \Psi _m  ^{\left( \frac{1}{m} \right)} \right) \le \rho (\Psi _1 \Psi _2 \cdots \Psi _m) ^{\frac{1}{m}}
\label{Hurho}
\ee
and
\be
\hat{\rho}  \left(\Psi _1 ^{\left( \frac{1}{m} \right)} \circ \cdots \circ \Psi _m  ^{\left( \frac{1}{m} \right)} \right) \le \hat{\rho}   (\Psi _1 \Psi _2 \cdots \Psi _m) ^{\frac{1}{m}}.
\label{Hurhoh}
\ee
\label{genHuDP}
\end{theorem}

\begin{corollary} Let $\Psi $ and $\Sigma$  be bounded sets of positive kernel operators on a Banach function space $L$.
Then we have
\be
\rho \left(\Psi ^{\left( \frac{1}{2} \right)} \circ \Sigma  ^{\left( \frac{1}{2} \right)} \right) \le \rho (\Psi \Sigma) ^{\frac{1}{2}}
\label{Hurho2}
\ee
and
\be
\hat{\rho}  \left(\Psi ^{\left( \frac{1}{2} \right)} \circ \Sigma  ^{\left( \frac{1}{2} \right)} \right) \le \hat{\rho} (\Psi \Sigma) ^{\frac{1}{2}}
\label{Hurhoh2}
\ee
\label{genSch}
\end{corollary}
Below we refine Corollary \ref{genSch}  in one of our main results (Theorem \ref{refin}) and in Corollary \ref{additional_powers}, while Theorem \ref{powers} refines Theorem \ref{family}. In these results we also establish that analogue results hold for the generalized and joint essential radii in the case when $L$ and $L^*$ have order continuous norms.  

\section{Results}

The following result is proved in a similar way as Theorems \ref{family} and \ref{genHuDP}  by replacing $\rho(\cdot)$ with $\rho _{ess}(\cdot)$ and $\|\cdot\|$ with $\beta (\cdot)$. To avoid too much repetition of ideas we omit the details of the proof.
\begin{theorem} 
Let $\Psi _1, \ldots \Psi _m$ be bounded sets of positive kernel operators on a Banach function space $L$ and let 
 $\alpha _1, \ldots \alpha _m$ be positive numbers such that \\
$\sum _{i=1} ^m \alpha _i = 1$. If $L$ and $L^*$ have order continuous norms then 
\be
r (\Psi _1 ^{( \alpha _1)} \circ \cdots \circ \Psi _m ^{(\alpha _m)} ) \le 
r(\Psi _1)^{ \alpha _1} \, \cdots r(\Psi _m)^{\alpha _m} 
\label{gsh3}
\ee
and
\be
r \left(\Psi _1 ^{\left( \frac{1}{m} \right)} \circ \cdots \circ \Psi _m  ^{\left( \frac{1}{m} \right)} \right) \le r(\Psi _1 \Psi _2 \cdots \Psi _m) ^{\frac{1}{m}}\;\;\;\;\;\;\;\;
\label{Hu_ess}
\ee
hold for each  $r\in \{ \rho _{ess}, \hat{\rho} _{ess}\}$.
\label{ess}
\end{theorem}
As already pointed out in \cite[Remark 3.4]{P17+}, the inequalities (\ref{genHuBfsP}) can be deduced from the proof of Theorem  \ref{genHuDP}. In a similar way the following result follows from the proof of (\ref{Hu_ess}). 
\begin{theorem}
Given a Banach function space $L$ such that $L$ and  $L^*$ have order continuous norms, let $A_1, A_2, \ldots , A_m$ be  positive kernel operators on $L$. If \\
$P_j = A_j \ldots A_m A_1 \ldots A_{j-1}$ for $j=1,\ldots , m$, then we have
$$\rho _{ess} \left(A_1^{\left(\frac{1}{m}\right)} \circ A_2^{\left(\frac{1}{m}\right)} \circ \cdots \circ 
A_m^{\left(\frac{1}{m}\right)}\right) \;\;\;\;\;\;\;\;\;\;\;\;\;\;\;\;\;\;\;\;\;\;\;\;\;\;\;\;\;\;\;\;\;\;$$
\be
\le \rho _{ess} \left(P_1^{\left(\frac{1}{m}\right)} \circ P_2^{\left(\frac{1}{m}\right)} \circ \cdots \circ 
P_m^{\left(\frac{1}{m}\right)}\right) ^{\frac{1}{m}}  \le \rho _{ess} (A_1 A_2 \cdots A_m) ^{\frac{1}{m}} .  
\label{genHuBfsP_ess}
\ee
\label{essPji}
\end{theorem}
\begin{corollary}
Given a Banach function space $L$ such that $L$ and  $L^*$ have order continuous norms, let $A$ and $B$ be  positive kernel operators on $L$.
Then
\be
\rho _{ess} \left(A^{(\frac{1}{2})} \circ B^{(\frac{1}{2})} \right) \le  \rho _{ess} \left((AB)^{(\frac{1}{2})} \circ (BA)^{(\frac{1}{2})}\right)^{\frac{1}{2}} \le \rho _{ess} (AB) ^{\frac{1}{2}}.
\label{sch_ess_ref}
\ee
\end{corollary}



Inequalities (\ref{sch_ess_ref}) are generalized and refined in Theorem \ref{refin}.
The following lemma follows from (\ref{basic2}) and commutativity of Hadamard product (or directly from Cauchy-Schwartz inequality).
\begin{lemma} Let $A,B,C,D$ be 
positive kernel operators on a Banach function space $L$.
Then we have
\be
(A^{(\frac{1}{2})} \circ B^{(\frac{1}{2})})(C^{(\frac{1}{2})} \circ D^{(\frac{1}{2})}) \le (AC)^{(\frac{1}{2})}  \circ (BD)^{(\frac{1}{2})},
\label{CS}
\ee
\be
(A^{(\frac{1}{2})} \circ B^{(\frac{1}{2})})(C^{(\frac{1}{2})} \circ D^{(\frac{1}{2})}) \le (AD)^{(\frac{1}{2})}  \circ (BC)^{(\frac{1}{2})}.
\label{CScomm}
\ee
\end{lemma}

The following result is a refinement of Corollary \ref{genSch} and it establishes  an analogue for the generalized essential and joint essential spectral radii. In the case of $r\in \{\rho , \hat{\rho}\}$ it is a kernel version of \cite[Theorems 3.5 and 3.7, Remark 3.10]{P12}.

\begin{theorem} Let $\Psi $ and $\Sigma$  be bounded sets of positive kernel operators on a Banach function space $L$.
If $r \in \{\rho, \hat{\rho}\}$ and $\alpha \in [0,1]$  then we have
$$r \left(\Psi ^{\left( \frac{1}{2} \right)} \circ \Sigma  ^{\left( \frac{1}{2} \right)} \right) \le  r \left((\Psi \Sigma)^{\left(\frac{1}{2}\right)} \circ (\Sigma \Psi)^{\left(\frac{1}{2}\right)}\right) ^{\frac{1}{2}}  $$
\be
\le   r \left((\Psi \Sigma)^{\left(\frac{1}{2}\right)} \circ (\Psi \Sigma)^{\left(\frac{1}{2}\right)}\right) ^{\frac{1}{4}}   r \left((\Sigma\Psi)^{\left(\frac{1}{2}\right)} \circ ( \Sigma \Psi)^{\left(\frac{1}{2}\right)}\right) ^{\frac{1}{4}} \le  r (\Psi \Sigma) ^{\frac{1}{2}},
\label{HZPep}
\ee
$$r \left(\Psi ^{\left( \frac{1}{2} \right)} \circ \Sigma  ^{\left( \frac{1}{2} \right)} \right) \le r \left( \left (\Psi ^{\left( \frac{1}{2} \right)} \circ \Psi  ^{\left( \frac{1}{2} \right)}\right) \left (\Sigma^{\left( \frac{1}{2} \right)} \circ \Sigma  ^{\left( \frac{1}{2} \right)}\right) \right)  ^{\frac{1}{2} }   $$
\be
\le   r \left((\Psi \Sigma)^{\left(\frac{1}{2}\right)} \circ (\Psi \Sigma)^{\left(\frac{1}{2}\right)}\right) ^{\frac{\alpha}{2}}   r \left((\Sigma\Psi)^{\left(\frac{1}{2}\right)} \circ ( \Sigma \Psi)^{\left(\frac{1}{2}\right)}\right) ^{\frac{1-\alpha}{2}} \le  r (\Psi \Sigma) ^{\frac{1}{2}}.
\label{AudSchPep}
\ee

If, in addition, $L$ and $L^*$ have order continuous norms, then (\ref{HZPep}) and (\ref{AudSchPep}) hold also for each $r\in \{ \rho _{ess}, \hat{\rho} _{ess}\}$.
\label{refin}
\end{theorem}
\begin{proof} Let  $r \in \{\rho, \hat{\rho}\}$. For the proof of the first inequality in (\ref{HZPep}) take $m \in \NN$ and $A \in  \left(\Psi ^{\left( \frac{1}{2} \right)} \circ \Sigma  ^{\left( \frac{1}{2} \right)} \right)^{2m}$. There exist $A_i \in \Psi$ and $B_i \in \Sigma$ for $i=1, \ldots, m$,  such that
$$A= (A^{(\frac{1}{2})} _1 \circ B^{(\frac{1}{2})}  _1 ) (A^{(\frac{1}{2})} _2 \circ B^{(\frac{1}{2})}  _2 ) \cdots (A^{(\frac{1}{2})} _{2m-1} \circ B^{(\frac{1}{2})}  _{2m-1 })(A^{(\frac{1}{2})} _{2m} \circ B^{(\frac{1}{2})}  _{2m} ).$$
By (\ref{CScomm}) we have $A\le B$ for
\be
B=  (C^{(\frac{1}{2})} _1 \circ D^{(\frac{1}{2})}  _1 ) \cdots (C^{(\frac{1}{2})} _m \circ D^{(\frac{1}{2})}  _m ) \in  \left((\Psi \Sigma)^{\left(\frac{1}{2}\right)} \circ (\Sigma \Psi)^{\left(\frac{1}{2}\right)}\right) ^m,
\label{B}
\ee
where $C_i = A_{2i-1}B_{2i} \in \Psi \Sigma$ and $D_i = B_{2i-1}A_{2i} \in \Sigma\Psi $ for $i=1, \ldots , m$.
Therefore $\rho(A)^{\frac{1}{2m}}\le \rho (B)^{\frac{1}{2m}}$ and $\|A\|^{\frac{1}{2m}}\le \|B\|^{\frac{1}{2m}}$ and so the first inequality in (\ref{HZPep}) follows.

For the proof of the second inequality  in (\ref{HZPep}) observe that by (\ref{spectral2}) and (\ref{norm2})
$$\rho(B) \le \rho(C_1 \cdots C_m)^{\frac{1}{2}} \rho(D_1 \cdots D_m )^{\frac{1}{2}} \; \mathrm{and}\; \|B\| \le \|C_1 \cdots C_m\|^{\frac{1}{2}} \|D_1 \cdots D_m \|^{\frac{1}{2}}$$
for each $ B \in \left((\Psi \Sigma)^{\left(\frac{1}{2}\right)} \circ (\Sigma \Psi)^{\left(\frac{1}{2}\right)}\right) ^m$ and $C_i \in \Psi \Sigma$, $D_i  \in \Sigma\Psi $ for $i=1, \ldots , m$ that satisfy (\ref{B}).
Since $C_1 \cdots C_m \in (\Psi \Sigma)^m \subset  \left((\Psi \Sigma)^{\left(\frac{1}{2}\right)} \circ (\Psi \Sigma)^{\left(\frac{1}{2}\right)}\right)^m $ and $D_1 \cdots D_m \in (\Sigma \Psi)^m \subset  \left((\Sigma\Psi )^{\left(\frac{1}{2}\right)} \circ (\Sigma \Psi)^{\left(\frac{1}{2}\right)}\right)^m $, the second inequality in (\ref{HZPep}) follows. Note that the above two inclusions follow from trivial identities $C_i = C_i ^{(\frac{1}{2})} \circ C_i ^{(\frac{1}{2})}$ and $D_i = D_i ^{(\frac{1}{2})} \circ D_i ^{(\frac{1}{2})}$.

The last inequality in  (\ref{HZPep}) follows from Theorem \ref{family} and the fact that $r(\Psi \Sigma)= r(\Sigma\Psi)$.

For the proof of the first inequality in (\ref{AudSchPep}) take $m \in \NN$ and $A \in  \left(\Psi ^{\left( \frac{1}{2} \right)} \circ \Sigma  ^{\left( \frac{1}{2} \right)} \right)^{2m}$. There exist $A_i \in \Psi$ and $B_i \in \Sigma$ for $i=1, \ldots, m$,  such that
$$A= (A^{(\frac{1}{2})} _1 \circ B^{(\frac{1}{2})}  _1 ) (A^{(\frac{1}{2})} _2 \circ B^{(\frac{1}{2})}  _2 ) \cdots (A^{(\frac{1}{2})} _{2m-1} \circ B^{(\frac{1}{2})}  _{2m-1 })(A^{(\frac{1}{2})} _{2m} \circ B^{(\frac{1}{2})}  _{2m} )$$
$$ \;\; \;\;= (A^{(\frac{1}{2})} _1 \circ B^{(\frac{1}{2})}  _1 ) (B^{(\frac{1}{2})} _2 \circ A^{(\frac{1}{2})}  _2 ) \cdots (A^{(\frac{1}{2})} _{2m-1} \circ B^{(\frac{1}{2})}  _{2m-1 })(B^{(\frac{1}{2})} _{2m} \circ A^{(\frac{1}{2})}  _{2m} ).$$
It follows by  (\ref{spectral2}) and (\ref{norm2}) that
$$\rho (A) \le \rho (A_1B_2A_3B_4 \cdots A_{2m-1}B_{2m})^{\frac{1}{2}} \rho (B_1A_2B_3A_4 \cdots B_{2m-1}A_{2m})^{\frac{1}{2}}.$$
Since $A_1B_2A_3B_4 \cdots A_{2m-1}B_{2m} \in  (\Psi \Sigma)^m \subset  \left( \left( \Psi^{\left(\frac{1}{2}\right)} \circ \Psi^{\left(\frac{1}{2}\right)} \right) \left( \Sigma ^{\left(\frac{1}{2}\right)} \circ \Sigma ^{\left(\frac{1}{2}\right)} \right) \right)^m $
and  $B_1A_2B_3A_4 \cdots B_{2m-1}A_{2m} \in  (\Sigma \Psi)^m \subset  \left( \left( \Sigma ^{\left(\frac{1}{2}\right)} \circ \Sigma ^{\left(\frac{1}{2}\right)} \right) \left( \Psi ^{\left(\frac{1}{2}\right)} \circ \Psi ^{\left(\frac{1}{2}\right)} \right) \right)^m $
this implies 
$$r \left(\Psi ^{\left( \frac{1}{2} \right)} \circ \Sigma  ^{\left( \frac{1}{2} \right)} \right) $$
$$\le r \left( \left (\Psi ^{\left( \frac{1}{2} \right)} \circ \Psi  ^{\left( \frac{1}{2} \right)}\right) \left (\Sigma^{\left( \frac{1}{2} \right)} \circ \Sigma  ^{\left( \frac{1}{2} \right)}\right) \right)  ^{\frac{1}{4} }    r \left( \left (\Sigma ^{\left( \frac{1}{2} \right)} \circ \Sigma  ^{\left( \frac{1}{2} \right)}\right) \left (\Psi^{\left( \frac{1}{2} \right)} \circ \Psi  ^{\left( \frac{1}{2} \right)}\right) \right)  ^{\frac{1}{4} }$$
$$=  r \left( \left (\Psi ^{\left( \frac{1}{2} \right)} \circ \Psi  ^{\left( \frac{1}{2} \right)}\right) \left (\Sigma^{\left( \frac{1}{2} \right)} \circ \Sigma  ^{\left( \frac{1}{2} \right)}\right) \right)  ^{\frac{1}{2} },  $$
which proves the first inequality in  (\ref{AudSchPep}).

To prove the second inequality in (\ref{AudSchPep}) we first prove that
\be
r \left( \left (\Psi ^{\left( \frac{1}{2} \right)} \circ \Psi  ^{\left( \frac{1}{2} \right)}\right) \left (\Sigma^{\left( \frac{1}{2} \right)} \circ \Sigma  ^{\left( \frac{1}{2} \right)}\right) \right)  
\le   r \left((\Psi \Sigma)^{\left(\frac{1}{2}\right)} \circ (\Psi \Sigma)^{\left(\frac{1}{2}\right)}\right)
\label{Sch_mist}
\ee
 Choose $m \in \NN$ and
 $B  \in  \left( \left (\Psi ^{\left( \frac{1}{2} \right)} \circ \Psi  ^{\left( \frac{1}{2} \right)}\right) \left (\Sigma^{\left( \frac{1}{2} \right)} \circ \Sigma  ^{\left( \frac{1}{2} \right)}\right) \right)^{m}$. There exist $A_i \in \Psi$ and $B_i \in \Sigma$ for $i=1, \ldots, 2m$,  such that
$$B= (A^{(\frac{1}{2})} _1 \circ A^{(\frac{1}{2})}  _2 ) (B^{(\frac{1}{2})} _1 \circ B^{(\frac{1}{2})}  _2 ) \cdots (A^{(\frac{1}{2})} _{2m-1} \circ A^{(\frac{1}{2})}  _{2m })(B^{(\frac{1}{2})} _{2m -1} \circ B^{(\frac{1}{2})}  _{2m} ).$$
By (\ref{CS}) 
$$B \le  \left((A_1 B_1)^{(\frac{1}{2})} \circ (A_2B_2)^{(\frac{1}{2})}  \right)  \cdots \left((A_{2m-1}B _{2m-1})^{(\frac{1}{2})}  \circ (A_{2m}B _{2m})^{(\frac{1}{2})}\right) =:C $$
and so $r(B)^{1/m} \le r(C)^{1/m}$ and  $\|B\|^{1/m} \le \|C\|^{1/m}$. Since
$C \in  \left((\Psi \Sigma)^{\left(\frac{1}{2}\right)} \circ (\Psi \Sigma)^{\left(\frac{1}{2}\right)}\right)^m$ this implies (\ref{Sch_mist}).

By interchanging the roles of $\Psi$ and $\Sigma$ in (\ref{Sch_mist}) it follows that
\be
r \left( \left (\Psi ^{\left( \frac{1}{2} \right)} \circ \Psi  ^{\left( \frac{1}{2} \right)}\right) \left (\Sigma^{\left( \frac{1}{2} \right)} \circ \Sigma  ^{\left( \frac{1}{2} \right)}\right) \right)  
\le   r \left((\Sigma\Psi)^{\left(\frac{1}{2}\right)} \circ (\Sigma \Psi)^{\left(\frac{1}{2}\right)}\right).
\label{Pep_from_Sch_mist}
\ee
Now the  the second inequality in (\ref{AudSchPep}) follows from (\ref{Sch_mist}) and  (\ref{Pep_from_Sch_mist}).

The last inequality in (\ref{AudSchPep}) follows from Theorem \ref{family} and the fact that $r(\Psi \Sigma)= r(\Sigma\Psi)$.

If $L$ and $L^*$ have order continuous norms, then by replacing $\rho(\cdot)$ with $\rho _{ess} (\cdot)$ and $\|\cdot\|$ with $\beta(\cdot)$ in the proof above one obtains that  (\ref{HZPep}) and (\ref{AudSchPep}) hold also for  each $r\in \{ \rho _{ess}, \hat{\rho} _{ess}\}$, which completes the proof.
\end{proof}

\begin{remark}{\rm Note that $ r \left((\Psi \Sigma)^{\left(\frac{1}{2}\right)} \circ (\Sigma \Psi)^{\left(\frac{1}{2}\right)}\right)$ and $ r \left( \left (\Psi ^{\left( \frac{1}{2} \right)} \circ \Psi  ^{\left( \frac{1}{2} \right)}\right) \left (\Sigma^{\left( \frac{1}{2} \right)} \circ \Sigma  ^{\left( \frac{1}{2} \right)}\right) \right) $ are in general  not comparable, and similarly for 
$ r \left((\Psi \Sigma)^{\left(\frac{1}{2}\right)} \circ (\Sigma \Psi)^{\left(\frac{1}{2}\right)}\right)$ and \\
$  r \left((\Psi \Sigma)^{\left(\frac{1}{2}\right)} \circ (\Psi \Sigma)^{\left(\frac{1}{2}\right)}\right)$, as \cite[Example 3.11]{P12} shows.
}
\end{remark}

\begin{remark} {\rm Under the conditions of Theorem \ref{refin} it holds also that
\be
 r \left((\Psi \Sigma)^{\left(\frac{1}{2}\right)} \circ (\Sigma \Psi)^{\left(\frac{1}{2}\right)}\right) \le  r \left(\Psi ^2 \Sigma ^2 \right)^{\frac{1}{2}}.
\label{squares}
\ee
Indeed, it follows from Corollary \ref{genSch} and Theorem \ref{ess} that
$$ r \left((\Psi \Sigma)^{\left(\frac{1}{2}\right)} \circ (\Sigma \Psi)^{\left(\frac{1}{2}\right)}\right) \le  r \left(\Psi\Sigma\Sigma\Psi \right)^{\frac{1}{2}} = r \left(\Psi ^2 \Sigma ^2 \right)^{\frac{1}{2}},$$
}
which proves (\ref{squares}). 

As already observed in \cite[Example 3.15]{P12}, the inequality (\ref{squares}) may in some cases be better than the second inequality in (\ref{HZPep}).
\end{remark}
The following result refines the inequalities (\ref{gsh}), (\ref{gsh2}) and  (\ref{gsh3}).
\begin{theorem} Let $\Psi _1, \ldots \Psi _m$ be bounded sets of positive kernel operators on a Banach function space $L$ and let 
 $\alpha _1, \ldots \alpha _m$ be positive numbers such that \\
$\sum _{i=1} ^m \alpha _i = 1$.  If  $r \in \{\rho, \hat{\rho}\}$ and $k \in \NN$ then
\be
r (\Psi _1 ^{( \alpha _1)} \circ \cdots \circ \Psi _m ^{(\alpha _m)} ) \le  r ((\Psi _1 ^k ) ^{( \alpha _1)} \circ \cdots \circ (\Psi _m ^k) ^{(\alpha _m)} ) ^{ \frac{1}{k}} \le
r(\Psi _1)^{ \alpha _1} \, \cdots r(\Psi _m)^{\alpha _m} 
\label{gsh_ref}
\ee

If, in addition, $L$ and $L^*$ have order continuous norms, then (\ref{gsh_ref}) holds also  for each $r\in \{ \rho _{ess}, \hat{\rho} _{ess}\}$.
\label{powers}
\end{theorem}
\begin{proof} Let $r \in \{\rho, \hat{\rho}\}$ and $k \in \NN$. To prove the first inequality in (\ref{gsh_ref}) take $ n \in \NN$ and $A \in (\Psi _1 ^{( \alpha _1)} \circ \cdots \circ \Psi _m ^{(\alpha _m)} )^{kn}$. Then  $A=A_1 A_2\cdots A_n$, where
$$A_i= \left(A_{i \,1 \,1}  ^{( \alpha _1)}  \circ A_{i\, 1\, 2}  ^{( \alpha _2)}  \circ \cdots \circ A_{i\, 1\, m}  ^{( \alpha _m)} \right)\left(A_{i\, 2\, 1}  ^{( \alpha _1)}  \circ A_{i\, 2\, 2}  ^{( \alpha _2)}  \circ \cdots \circ A_{i\, 2\, m}  ^{( \alpha _m)} \right) \cdots $$
$$\cdots \left(A_{i k 1}  ^{( \alpha _1)}  \circ A_{i k 2}  ^{( \alpha _2)}  \circ \cdots \circ A_{i k m}  ^{( \alpha _m)} \right) $$
for some $A_{i\, j\, 1} \in \Psi _1, \ldots, A_{i\, j \,m} \in \Psi _m$ and all $j=1, \ldots, k$, $i=1, \ldots, n$. By Theorem \ref{DPBfs} it follows that  $A\le B=B_1 \cdots B_n$, where
$$A_i \le B_i:= \left(A_{i \,1 \,1} A_{i\, 2\, 1} \cdots A_{i k 1} \right) ^{( \alpha _1)} \circ \cdots \circ \left(A_{i \,1 \,m} A_{i\, 2\, m} \cdots A_{i k m} \right) ^{( \alpha _m)}. $$
Since $B \in  ((\Psi _1 ^k)^{( \alpha _1)} \circ \cdots \circ (\Psi _m ^k)^{(\alpha _m)} )^{n} $, $\rho(A)^{\frac{1}{kn}} \le \rho(B)^{\frac{1}{kn}}$ and $\|A\|^{\frac{1}{kn}} \le \|B\|^{\frac{1}{kn}}$, this proves the first inequality in (\ref{gsh_ref}).  

The second inequality in (\ref{gsh_ref}) follows from Theorem \ref{family} and the fact that $r(\Psi _i ^k)=r(\Psi _i) ^k$ for all 
$i=1,\ldots ,m$.

If $L$ and $L^*$ have order continuous norms,  then by replacing $\rho(\cdot)$ with $\rho _{ess} (\cdot)$ and $\|\cdot\|$ with $\beta(\cdot)$ in the proof above one obtains that  (\ref{gsh_ref}) holds also for  each $r\in \{ \rho _{ess}, \hat{\rho} _{ess}\}$, which completes the proof.
\end{proof}
\begin{corollary}  Let $\Psi$ and $\Sigma$ be bounded sets of positive kernel operators on a Banach function space $L$. 
If  $r \in \{\rho, \hat{\rho}\}$ and $k \in \NN$ then
\be
r \left(\Psi ^{\left( \frac{1}{2} \right)} \circ \Sigma ^{\left( \frac{1}{2} \right)} \right) \le  r \left((\Psi  ^k ) ^{\left( \frac{1}{2} \right)}\circ (\Sigma ^k) ^{\left( \frac{1}{2} \right)} \right) ^{ \frac{1}{k}} \le
r(\Psi )^{ \frac{1}{2} } r(\Sigma)^{\frac{1}{2} }. 
\label{gsh_ref_two}
\ee

If, in addition, $L$ and $L^*$ have order continuous norms, then (\ref{gsh_ref_two}) holds also  for each $r\in \{ \rho _{ess}, \hat{\rho} _{ess}\}$.
\label{k_refinement}
\end{corollary}
Even in the case of single operators the following refinement of Theorem \ref{special_case} appears to be new.
\begin{corollary}
Let $A_1, \ldots A _m$ be positive kernel operators on a Banach function space $L$ and let 
 $\alpha _1, \ldots \alpha _m$ be positive numbers such that 
$\sum _{i=1} ^m \alpha _i = 1$.  If  $k \in \NN$ then
\be
\rho (A_1 ^{( \alpha _1)} \circ \cdots \circ A_m ^{(\alpha _m)} ) \le  \rho ((A_1 ^k ) ^{( \alpha _1)} \circ \cdots \circ (A_m ^k) ^{(\alpha _m)} ) ^{ \frac{1}{k}} \le
\rho(A_1)^{ \alpha _1} \, \cdots \rho(A_m)^{\alpha _m} 
\label{spr_ref}
\ee

If, in addition, $L$ and $L^*$ have order continuous norms, then 
\be
\rho _{ess} (A_1 ^{( \alpha _1)} \circ \cdots \circ A_m ^{(\alpha _m)} ) \le  \rho _{ess} ((A_1 ^k ) ^{( \alpha _1)} \circ \cdots \circ (A_m ^k) ^{(\alpha _m)} ) ^{ \frac{1}{k}} \le
\rho _{ess}(A_1)^{ \alpha _1} \, \cdots \rho _{ess}(A_m)^{\alpha _m} 
\label{ess_spr_ref}
\ee
\end{corollary}
By applying Corollary \ref{k_refinement}, additional refinements of Theorem \ref{refin} are possible. We point out the following refinement of the last inequalities in (\ref{HZPep}) and (\ref{AudSchPep}).
\begin{corollary} Let $L$, $\Psi$, $\Sigma$, $\alpha$ and $r$ be as in Theorem \ref{refin}. Then
$$   r \left((\Psi \Sigma)^{\left(\frac{1}{2}\right)} \circ (\Psi \Sigma)^{\left(\frac{1}{2}\right)}\right) ^{\frac{\alpha}{2}}   r \left((\Sigma\Psi)^{\left(\frac{1}{2}\right)} \circ ( \Sigma \Psi)^{\left(\frac{1}{2}\right)}\right) ^{\frac{1-\alpha}{2}}$$
$$ \le  r \left(((\Psi \Sigma)^k)^{\left(\frac{1}{2}\right)} \circ ((\Psi \Sigma)^k)^{\left(\frac{1}{2}\right)}\right) ^{\frac{\alpha}{2k}}   r \left(((\Sigma\Psi)^k)^{\left(\frac{1}{2}\right)} \circ ( (\Sigma \Psi)^k)^{\left(\frac{1}{2}\right)}\right) ^{\frac{1-\alpha}{2k}}
 \le  r (\Psi \Sigma) ^{\frac{1}{2}}.$$
for all $k \in \NN$.  
\label{additional_powers}
\end{corollary}

\noindent {\bf Acknowledgements.} 
The author acknowledges a partial support of  the Slovenian Research Agency (grants P1-0222 and J1-8133).

\bibliographystyle{amsplain}

\end{document}